\documentclass[11pt]{article}

\usepackage[T1]{fontenc}
\usepackage{lmodern}

\usepackage{soul}
\usepackage{amsmath, amssymb, amsthm} 
\usepackage[margin=1in]{geometry}
\usepackage{color, graphicx}

\usepackage{etoolbox} 
\makeatletter
\patchcmd{\@maketitle}{\LARGE \@title}{\LARGE\bfseries\@title}{}{}

\renewcommand{\@seccntformat}[1]{\csname the#1\endcsname.\quad}
\makeatother

\definecolor{darkblue}{rgb}{0,0,.5}

\usepackage{hyperref}
\hypersetup{
	colorlinks=true,        	
	linkcolor=darkblue,         
	citecolor=darkblue,         
	urlcolor=darkblue          	
}

\makeatletter
\def\th@plain{%
	\thm@notefont{}
	\itshape 
}
\def\th@definition{%
	\thm@notefont{}
	\normalfont 
}

\renewenvironment{proof}[1][\proofname]{\par
	\normalfont
	\topsep0\p@\@plus3\p@ \trivlist
	\item[\hskip\labelsep\itshape
	#1\@addpunct{.}]\ignorespaces
}{%
	\qed\endtrivlist
}
\makeatother

\newtheorem{theorem}{Theorem}[section]
\newtheorem{lemma}[theorem]{Lemma}
\newtheorem{corollary}[theorem]{Corollary}
\newtheorem{proposition}[theorem]{Proposition}

\theoremstyle{definition}

\theoremstyle{definition}

\theoremstyle{definition}
\newtheorem{remark}[theorem]{Remark}

\usepackage[shortlabels]{enumitem}

\allowdisplaybreaks

\newcommand{\menge}[2]{\{{#1}~\big |~{#2}\}} 
 
\newcommand{\Menge}[2]{\left\{{#1}~\Big|~{#2}\right\}}

\newcommand{\scal}[2]{\left\langle {#1},{#2} \right\rangle}

\newcommand{\NN}{\ensuremath{\mathbb N}}
\newcommand{\nnn}{\ensuremath{{n\in{\mathbb N}}}}
\newcommand{\RR}{\ensuremath{\mathbb R}}
\newcommand{\RP}{\ensuremath{\mathbb{R}_+}}
\newcommand{\RPP}{\ensuremath{\mathbb{R}_{++}}}

\newcommand{\argmin}{\ensuremath{\operatorname*{argmin}}}

\newcommand{\zer}{\ensuremath{\operatorname{zer}}}
\newcommand{\dom}{\ensuremath{\operatorname{dom}}}

\newcommand{\gra}{\ensuremath{\operatorname{gra}}}
\newcommand{\Fix}{\ensuremath{\operatorname{Fix}}}
\newcommand{\Id}{\ensuremath{\operatorname{Id}}}
\newcommand{\prox}{\ensuremath{\operatorname{Prox}}}

\begin{document}

\title{\sf An adaptive splitting algorithm for the sum of three operators}

\author{
Minh N.\ Dao\thanks{
School of Engineering, Information Technology and Physical Sciences, Federation University, Australia.
E-mail: \texttt{m.dao@federation.edu.au}}
~and~
Hung M.\ Phan\thanks{Department of Mathematical Sciences, Kennedy College of Sciences, University of Massachusetts Lowell, Lowell, MA 01854, USA. E-mail: \texttt{hung\char`_phan@uml.edu}.}}

\date{April 12, 2021}

\maketitle

\begin{abstract}
Splitting algorithms for finding a zero of sum of operators often involve multiple steps which are referred to as forward or backward steps. Forward steps are the explicit use of the operators and backward steps involve the operators implicitly via their resolvents. In this paper, we study an adaptive splitting algorithm for finding a zero of the sum of three operators. We assume that two of the operators are generalized monotone and their resolvents are computable, while the other operator is cocoercive but its resolvent is missing or costly to compute. Our splitting algorithm adapts new parameters to the generalized monotonicity of the operators and, at the same time, combines appropriate forward and backward steps to guarantee convergence to a solution of the problem.
\end{abstract}


\smallskip
\noindent{\bf Mathematics Subject Classification (MSC 2020):} 
47H05, 
65K05, 
65K10, 
90C25. 

\section{Introduction}

Operator splitting algorithms are developed for structured optimization problems based on the idea of performing the computation separately on individual operators. At each iterations, they requires multiple steps which are known as either \emph{forward} or \emph{backward} steps. The forward steps are almost always easy as they use the operator directly. The backward steps, on the other hand, are often more complicated as they use the resolvents of the operators. While there are many operators whose resolvents are readily computable, there exist operators whose resolvents may not be computable in closed form, thus, it is necessary to use the forward steps in certain situations. Notable examples of splitting algorithms include the \emph{forward-backward algorithm} \cite{LM79} and the \emph{Douglas--Rachford algorithm} \cite{DR56,LM79}, and many others.

In this paper, we study an adaptive splitting algorithm for the inclusion problem 
\begin{equation}\label{e:zero_ABC}
\text{find $x\in X$ such that~} 0\in Ax+Bx+Cx,
\end{equation}
where $X$ is a real Hilbert space, $A,B\colon X\rightrightarrows X$ are generalized monotone operators, and $C\colon X\to X$ is a cocoercive operator. It is worth mentioning that the problem of finding a zero of the sum of finitely many maximally monotone operators and a cocoercive operator can be written as finding a zero of the sum of three operators in a product space: a maximally monotone operator, a normal cone of a closed subspace, and a cocoercive operator \cite{Rag19,RFP13}. The resulting problem can be then solved by the so-called \emph{forward-Douglas--Rachford splitting algorithm} \cite{Bri15,Rag19}. In \cite{DY17}, a \emph{three-operator splitting algorithm} is proposed for solving \eqref{e:zero_ABC} in the case when $A$ and $B$ are maximally monotone. This algorithm can be seen as a generalization of the forward-Douglas--Rachford splitting algorithm. When one of the three operators is zero, the algorithm in \cite{DY17} reduces to the Douglas--Rachford algorithm, forward-backward algorithm, and also \emph{backward-forward algorithm} which has recently been studied in \cite{APR18}. 

The main contribution of the paper is to develop an adaptive splitting algorithm for solving \eqref{e:zero_ABC} when $A$ and $B$ are strongly and weakly monotone operators and $C$ is a cocoercive operator. We utilize new parameters so that the generated sequence converges weakly to a fixed point, while the corresponding image sequence via the resolvent (a.k.a. the shadow sequence) converges weakly to a solution of the original problem. If the strong monotonicity outweighs the weak monotonicity, the convergence of the shadow sequence is strong. In addition, we recover some classical results for the forward-backward, backward-forward, and Douglas--Rachford algorithms. An application to minimizing the sum of three functions is also included.

On the one hand, our new algorithm enhances the framework of \cite{DY17} to allow for handling generalized monotone operators. On the other hand, it extends the adaptive approach in \cite{DP19} to incorporate the third operator that is cocoercive and whose resolvent might not be explicitly computable. In particular, \cite{DP19} proposed the \emph{adaptive Douglas--Rachford splitting algorithm} for finding a zero of the sum of two operators, one of which is strongly monotone and the other is weakly monotone. Therein, adaptive parameters were used to accommodate the corresponding monotonicity properties of the operators. This approach was later studied in \cite{BDP19} using conically averagedness, which is indeed a very different perspective.

On another note, it is well known that the alternating direction method of multipliers (ADMM) can be written as the Douglas--Rachford algorithm in dual settings. Recently, this important relation has been extended in \cite{BCP21} for the adaptive framework, namely, a new adaptive ADMM can be written as the adaptive Douglas--Rachford algorithm in dual settings. We refer interested readers to \cite{BCP21} for a rather comprehensive discussion on the ADMM.

The remainder of the paper is organized as follows. In Section~\ref{s:prelim}, we present our adaptive splitting algorithm and recall some preliminary materials. Section~\ref{s:abstrt} provides an abstract convergence result, which will be used to derive the main results in Section~\ref{s:3opers}. In Section~\ref{s:2opers}, we revisit some convergence results for the case of two operators based on the newly developed framework. Finally, Section~\ref{s:min3f} presents an immediate application of the main results to minimizing the sum of three functions.

\section{The algorithm}
\label{s:prelim}

Throughout, $X$ is a real Hilbert space with inner product $\scal{\cdot}{\cdot}$ and induced norm $\|\cdot\|$. 
The set of nonnegative integers is denoted by $\NN$ and the set of real numbers is denoted by $\RR$. We denote the set of nonnegative real numbers by $\RP := \menge{x \in \RR}{x \geq 0}$ and the set of the positive real numbers by $\RPP := \menge{x \in \RR}{x >0}$. The notation $A\colon X\rightrightarrows X$ to indicate that $A$ is a set-valued operator on $X$ and the notation $A\colon X\to X$ is to indicate that $A$ is a single-valued operator on $X$. 

Let $A\colon X\rightrightarrows X$ be an operator on $X$. Then its \emph{domain} is $\dom A :=\menge{x\in X}{Ax\neq \varnothing}$, its set of \emph{zeros} is $\zer A :=\menge{x\in X}{0\in Ax}$, and its set of \emph{fixed points} is $\Fix A :=\menge{x\in X}{x\in Ax}$. The \emph{graph} of $A$ is the set $\gra A :=\menge{(x,u)\in X\times X}{u\in Ax}$ and the \emph{inverse} of $A$, denoted by $A^{-1}$, is the operator with graph $\gra A^{-1} :=\menge{(u,x)\in X\times X}{u\in Ax}$. The \emph{resolvent} of $A$ is defined by
\begin{equation}
J_A :=(\Id+A)^{-1},
\end{equation}
where $\Id$ is the \emph{identity operator}.

Now, let $\eta,\gamma,\delta\in \RPP$ and set $\lambda :=1+\frac{\delta}{\gamma}$. In order to address problem \eqref{e:zero_ABC}, we employ the operator
\begin{equation}\label{e:TABC}
T_{A,B,C} :=\Id -\eta J_{\gamma A} +\eta J_{\delta B}\big((1-\lambda)\Id +\lambda J_{\gamma A} -\delta CJ_{\gamma A}\big).
\end{equation}
We will also refer to $\gamma$ and $\delta$ as the \emph{resolvent parameters} as they are used to \emph{scale} the operators $A,B$ in their respective resolvents. In fact, we adapt $\gamma$ and $\delta$ to the generalized monotonicity of $A$ and $B$ in order to guarantee the convergence of $T_{A,B,C}$. Intuitively, in the case $A$ and $B$ are maximally monotone, one would expect the use of equal resolvent parameters $\gamma=\delta$, and in other cases, $\gamma$ are $\delta$ are no longer the same. This phenomenon was initially observed in \cite{DP19,DP20}. Although the imbalance of monotonicity can be resolved by shifting the identity between the operators as in \cite[Remark~4.15]{DP19}, our plan is to conduct the convergence analysis of the algorithm applied to the \emph{original} operators.

To motivate the use of \eqref{e:TABC}, the following result shows the relationship between the fixed point set of $T_{A,B,C}$ and the solution set of \eqref{e:zero_ABC}.

\begin{proposition}[fixed points of $T_{A,B,C}$]
\label{p:fix}
Let $T_{A,B,C}$ be defined by \eqref{e:TABC}.
Then $\Fix T_{A,B,C}\neq \varnothing$ if and only if $\zer(A+B+C)\neq\varnothing$. Moreover, if $J_{\gamma A}$ is single-valued, then 
\begin{equation}
J_{\gamma A}(\Fix T_{A,B,C}) =\zer(A+B+C).
\end{equation} 
\end{proposition}
\begin{proof}
Let $x\in \dom T_{A,B,C}$. We have 
\begin{equation}
T_{A,B,C} x =\menge{x -\eta a +\eta J_{(\lambda-1)\gamma B}\big((1-\lambda)x +\lambda a -(\lambda-1)\gamma Ca\big)}{a\in J_{\gamma A}x}.
\end{equation}
Therefore,
\begin{subequations}
\begin{align}
x\in \Fix T_{A,B,C}
&\iff \exists a\in J_{\gamma A}x,\quad a\in J_{(\lambda-1)\gamma B}\big((1-\lambda)x +\lambda a -(\lambda-1)\gamma Ca\big)\\
&\iff \exists a\in J_{\gamma A}x,\quad (1-\lambda)x +\lambda a -(\lambda-1)\gamma Ca -a\in (\lambda-1)\gamma Ba\\
&\iff \exists a\in X,\quad x-a\in \gamma Aa \text{~and~} a-x\in Ba+Ca\\
&\iff \exists a\in J_{\gamma A}x\cap \zer(A+B+C), 
\end{align}
\end{subequations}
which completes the proof.
\end{proof}

In the rest of this section, we recall some preliminary concepts and results. Let $T\colon X\to X$ be a single-valued operator on $X$. Then $T$ is \emph{nonexpansive} if it is Lipschitz continuous with constant $1$ on its domain, i.e.,
\begin{equation}
\forall x, y\in \dom T,\quad \|Tx-Ty\|\leq \|x-y\|.
\end{equation}
The operator $T$ is said to be \emph{conically averaged} with constant $\theta\in \RPP$ (see \cite{BDP19, BMW20}), if there exists a nonexpansive operator $N\colon X\to X$ such that 
\begin{equation}
T =(1-\theta)\Id +\theta N.
\end{equation}
Given a conically $\theta$-averaged operator, it is $\theta$-averaged when $\theta\in \left]0,1\right[$, and nonexpansive when $\theta =1$. Some more properties on conical averagedness are discussed in the following proposition.

\begin{proposition}
\label{p:coavg}
Let $T\colon X\to X$, $\theta\in \RPP$, and $\lambda\in \RPP$. Then the following are equivalent:
\begin{enumerate}
\item\label{p:coavg_ori}
$T$ is conically $\theta$-averaged.
\item\label{p:coavg_rlx} 
$(1-\lambda)\Id +\lambda T$ is conically $\lambda\theta$-averaged.  
\item\label{p:coavg_ine} 
For all $x,y\in \dom T$, 
\begin{equation}
\|Tx-Ty\|^2\leq \|x-y\|^2 -\left(\frac{1}{\theta}-1\right)\|(\Id-T)x -(\Id-T)y\|^2.
\end{equation}
\end{enumerate}
\end{proposition}
\begin{proof}
See \cite[Proposition~2.2]{BDP19}.
\end{proof}

Recall from \cite{DP19} that an operator $A\colon X\rightrightarrows X$ is \emph{$\alpha$-monotone} with $\alpha\in \RR$ if
\begin{equation}
\forall (x,u), (y,v)\in\gra A,\quad
\scal{x-y}{u-v}\geq \alpha\|x-y\|^2.
\end{equation}
We say that $A$ is \emph{monotone} if $\alpha =0$, \emph{strongly monotone} if $\alpha >0$, and \emph{weakly monotone} if $\alpha <0$. The operator $A$ is said to be \emph{maximally $\alpha$-monotone} if it is $\alpha$-monotone and there is no $\alpha$-monotone operator $B\colon X\rightrightarrows X$ such that $\gra B$ properly contains $\gra A$.

We say that $A$ is $\sigma$-\emph{cocoercive} if $\sigma\in\RPP$ and
\begin{equation}
\forall (x,u), (y,v)\in\gra A,\quad
\scal{x-y}{u-v}\geq \sigma\|u-v\|^2.
\end{equation}
Clearly, if $A$ is $\sigma$-cocoercive, then $A$ is single-valued and monotone. In fact, $\sigma$-cocoercivity was extended to $\sigma$-comonotonicity to allows for negative parameter $\sigma$, see \cite{BDP19,BMW20} for more details.

\begin{proposition}[single-valued and full domain]
\label{p:resol-mono}
Let $A\colon X\rightrightarrows X$ be $\alpha$-monotone and let $\gamma\in\RPP$ such that $1+\gamma\alpha >0$. Then the following hold:
\begin{enumerate}
\item\label{p:resol-mono_single} 
$J_{\gamma A}$ is single-valued and $(1+\gamma\alpha)$-cocoercive. 
\item 
$\dom J_{\gamma A}=X$ if and only if $A$ is maximally $\alpha$-monotone.
\end{enumerate}
\end{proposition}
\begin{proof}
See \cite[Lemma~3.3 and Proposition~3.4]{DP19}.
\end{proof}

Finally, we recall the \emph{demiclosedness principle} for cocoercive operators developed in \cite{BCP20}. A fundamental result in the theory of nonexpansive mapping is Browder's celebrated demiclosedness principle \cite{Bro68}. It was extended for finitely many firmly nonexpansive mappings in \cite{Bau13}, and was later generalized in \cite{BCP20} for a finite family of conically averaged mappings or for a finite family of cocoercive mappings. An instant application of the demiclosedness principles is to provide a simple proof for the weak convergence of the shadow sequence of the Douglas--Rachford algorithm \cite{Bau13}, and of the adaptive Douglas--Rachford algorithm \cite{BCP20}. For our analysis, we recall only the result for two operators.

\begin{proposition}[demiclosedness principle for balanced cocoercive operators]
\label{p:demi}
Let $T_1: X\to X$ and $T_2: X\to X$ be respectively $\sigma_1$- and $\sigma_2$-cocoercive, let $(x_n)_{n\in \mathbb{N}}$ and $(z_n)_{n\in \mathbb{N}}$ be sequences in $X$, and let $\rho_1,\rho_2\in\RPP$ be such that
\begin{equation}
\frac{\rho_1\sigma_1+\rho_2\sigma_2}{\rho_1+\rho_2}\geq 1.
\end{equation}
Suppose that as $n\to +\infty$,
\begin{subequations}
\begin{align}
x_n\rightharpoonup x^*,\quad z_n\rightharpoonup z^*,\\
T_1x_n\rightharpoonup y^*,\quad T_2z_n\rightharpoonup y^*,\\
\rho_1(x_n-T_1x_n) + \rho_2(z_n-T_2z_n)\to \rho_1(x^*-y^*)+\rho_2(z^*-y^*).
\end{align}
\end{subequations}
Then $y^* =T_1x^* =T_2z^*$.
\end{proposition}
\begin{proof}
Apply~\cite[Theorem~3.2]{BCP20} for two operators.
\end{proof}

\section{An abstract convergence result}
\label{s:abstrt}

In order to study $T_{A,B,C}$, it is reasonable to consider the general operator
\begin{equation}\label{e:T}
T :=\Id -\eta T_1 +\eta T_2(-\nu\Id +\lambda T_1 -\delta T_3T_1),
\end{equation}
where $T_1,T_2,T_3\colon X\to X$ and $\eta,\nu,\lambda,\delta\in \RPP$. In this section, we establish a convergence result for the operator $T$ under the cocoercivity of $T_1,T_2,T_3$. We begin with a useful technical lemma.

\begin{lemma}
\label{l:abc}
Let $a,b,c,d$ be in $X$ and let $\eta,\nu, \lambda, \delta$ be in $\RPP$. Set $e :=-\nu a +\lambda b -\delta c$ and $f :=a -\eta b +\eta d$. Then, for all $\sigma\in \RPP$,
\begin{align}
\|f\|^2 &=\|a\|^2 -\left(\frac{\lambda}{\eta\nu} -\frac{\delta}{2\eta\nu\sigma} -1\right)\|a-f\|^2 -\frac{\delta}{2\eta\nu\sigma}\|a-f -2\eta\sigma c\|^2 \notag \\
&\qquad +\frac{\lambda\eta}{\nu}\|b\|^2 +\frac{\lambda\eta}{\nu}\|d\|^2 -2\eta\scal{a}{b} -2\frac{\eta}{\nu}\scal{e}{d} -\frac{2\delta\eta}{\nu}(\scal{c}{b} -\sigma\|c\|^2).
\end{align}
\end{lemma}
\begin{proof}
By assumption, $a-f =\eta(b-d)$ and $\lambda b =\nu a +\delta c +e$, which imply that  
\begin{subequations}\label{e:a-f}
\begin{align}
\frac{\lambda}{\eta^2}\|a-f\|^2 &=\lambda\|b-d\|^2 \\
&=\lambda\|b\|^2 +\lambda\|d\|^2 -2\lambda\scal{b}{d} \\
&=\lambda\|b\|^2 +\lambda\|d\|^2 -2\scal{\nu a +\delta c +e}{d} \\
&=\lambda\|b\|^2 +\lambda\|d\|^2 -2\nu\scal{a}{d} -2\delta\scal{c}{d} -2\scal{e}{d}. 
\end{align}
\end{subequations}
Writing $d =b -\frac{1}{\eta}(a-f)$, we have that
\begin{equation}\label{e:ad}
-2\nu\scal{a}{d}=-2\nu\scal{a}{b} +\frac{2\nu}{\eta}\scal{a}{a-f}
=-2\nu\scal{a}{b} +\frac{\nu}{\eta}(\|a\|^2 +\|a-f\|^2 -\|f\|^2)
\end{equation}
and that
\begin{subequations}\label{e:cd}
\begin{align}
-2\delta\scal{c}{d} &= -2\delta\scal{c}{b} +\frac{2\delta}{\eta}\scal{c}{a-f} \\
&=-2\delta(\scal{c}{b} -\sigma\|c\|^2)
- 2\delta\sigma\|c\|^2 +\frac{2\delta}{\eta}\scal{c}{a-f}\\
&= -2\delta(\scal{c}{b} -\sigma\|c\|^2)
-\frac{\delta}{2\eta^2\sigma}\|a-f -2\eta\sigma c\|^2 +\frac{\delta}{2\eta^2\sigma}\|a-f\|^2. 
\end{align}
\end{subequations}
Substituting \eqref{e:ad} and \eqref{e:cd} into \eqref{e:a-f} yields
\begin{align}
\frac{\nu}{\eta}\|f\|^2 &=\frac{\nu}{\eta}\|a\|^2 -\left(\frac{\lambda}{\eta^2} -\frac{\delta}{2\eta^2\sigma} -\frac{\nu}{\eta}\right)\|a-f\|^2 -\frac{\delta}{2\eta^2\sigma}\|a-f -2\eta\sigma c\|^2 \notag \\
&\qquad +\lambda\|b\|^2 +\lambda\|d\|^2 -2\nu\scal{a}{b} -2\scal{e}{d} -2\delta(\scal{c}{b} -\sigma\|c\|^2),
\end{align}
which implies the conclusion.
\end{proof}

The following proposition is inspired by \cite[Proposition~2.1]{DY17}.

\begin{proposition}
\label{p:abs}
Let $T_1$, $T_2$, and $T_3$ be respectively $\sigma_1$-, $\sigma_2$-, and $\sigma_3$-cocoercive. Let $\eta, \nu,\lambda,\delta\in \RPP$ and define
\begin{equation}
T :=\Id -\eta T_1 +\eta T_2(-\nu\Id +\lambda T_1 -\delta T_3T_1).
\end{equation}
Then the following hold:
\begin{enumerate}
\item\label{p:abs_lambda=} 
If $\lambda =2\nu\sigma_1 =2\sigma_2$ and 
\begin{equation}
\eta^* :=\frac{1}{\nu}\left(\lambda-\frac{\delta}{2\sigma_3}\right)>0,
\end{equation}
then, for all $x,y\in \dom T$, 
\begin{align}\label{e:lambda=}
\|Tx-Ty\|^2&\leq \|x-y\|^2 -\left(\frac{\eta^*}{\eta}-1\right)\|(\Id-T)x-(\Id-T)y\|^2 \notag\\
&\quad -\frac{\delta}{2\eta\nu\sigma_3}\|(\Id-T)x-(\Id-T)y-2\eta\sigma_3(T_3T_1x-T_3T_1y)\|^2.
\end{align}

\item\label{p:abs_lambda<} 
If $\lambda <\nu\sigma_1+\sigma_2$ and 
\begin{equation}
\eta^* :=\frac{1}{\nu}\left(\frac{(2\nu\sigma_1-\lambda)(2\sigma_2-\lambda)}{2(\nu\sigma_1+\sigma_2-\lambda)}+\lambda-\frac{\delta}{2\sigma_3}\right)>0,
\end{equation} 
then, for all $x,y\in \dom T$, 
\begin{align}\label{e:lambda<}
\|Tx-Ty\|^2&\leq \|x-y\|^2 -\left(\frac{\eta^*}{\eta}-1\right)\|(\Id-T)x-(\Id-T)y\|^2 \notag\\
&-\frac{\delta}{2\eta\nu\sigma_3}\|(\Id-T)x-(\Id-T)y-2\eta\sigma_3(T_3T_1x-T_3T_1y)\|^2\notag\\
&-\frac{\eta}{2\nu(\nu\sigma_1+\sigma_2-\lambda)} \|(2\nu\sigma_1-\lambda)(T_1x-T_1y) +(2\sigma_2-\lambda)(T_2Sx-T_2Sy)\|^2,
\end{align}
where $S:=-\nu\Id+\lambda T_1-\delta T_3T_1$.
\end{enumerate}
In both cases, $T$ is conically $\frac{\eta}{\eta^*}$-averaged.
\end{proposition}
\begin{proof}
Let $x,y\in X$ be arbitrary and set $S :=-\nu\Id +\lambda T_1 -\delta T_3T_1$. Then $T =\Id -\eta T_1 +\eta T_2S$. Define
\begin{subequations}
\begin{align}
a &:=x-y, &b &:=T_1x-T_1y,\\
c &:=T_3T_1x-T_3T_1y, &d &:=T_2Sx-T_2Sy,\\
e &:=Sx-Sy, &f &:=Tx-Ty.
\end{align}
\end{subequations}
Then $e =-\nu a +\lambda b -\delta c$ and $f =a -\eta b +\eta d$. By applying Lemma~\ref{l:abc} (with $\sigma =\sigma_3$),
\begin{align}\label{e:abc}
\|f\|^2 &=\|a\|^2 -\left(\frac{\lambda}{\eta\nu} -\frac{\delta}{2\eta\nu\sigma_3} -1\right)\|a-f\|^2 -\frac{\delta}{2\eta\nu\sigma_3}\|a-f -2\eta\sigma c\|^2 \notag \\
&\qquad +\frac{\lambda\eta}{\nu}\|b\|^2 +\frac{\lambda\eta}{\nu}\|d\|^2 -2\eta\scal{a}{b} -2\frac{\eta}{\nu}\scal{e}{d} -\frac{2\delta\eta}{\nu}(\scal{c}{b} -\sigma_3\|c\|^2).
\end{align}
On the other hand, the cocoercivity of $T_1$, $T_2$, and $T_3$ yields
\begin{equation}
\scal{a}{b}\geq\sigma_1\|b\|^2,\quad
\scal{e}{d}\geq\sigma_2\|d\|^2,\quad
\scal{b}{c}\geq\sigma_3\|c\|^2.
\end{equation}
Combining with \eqref{e:abc}, we obtain that
\begin{align}\label{e:abc'}
\|f\|^2 &\leq\|a\|^2 -\left(\frac{\lambda}{\eta\nu} -\frac{\delta}{2\eta\nu\sigma_3} -1\right)\|a-f\|^2 -\frac{\delta}{2\eta\nu\sigma_3}\|a-f -2\sigma_3 c\|^2 \notag \\
&\qquad -\frac{\eta}{\nu}(2\nu\sigma_1-\lambda)\|b\|^2 -\frac{\eta}{\nu}(2\sigma_2-\lambda)\|d\|^2.
\end{align}

\ref{p:abs_lambda=}: Since $\lambda=2\nu\sigma_1=2\sigma_2$, \eqref{e:abc'} reduces to 
\begin{align}
\|f\|^2 &\leq\|a\|^2 -\left(\frac{\eta^*}{\eta} -1\right)\|a-f\|^2 -\frac{\delta}{2\eta\nu\sigma_3}\|a-f -2\sigma_3 c\|^2,
\end{align}
which gives \eqref{e:lambda=}.

\ref{p:abs_lambda<}: Set $\kappa :=2\nu\sigma_1-\lambda$ and $\mu :=2\sigma_2-\lambda$. Then $\kappa+\mu =2(\nu\sigma_1+\sigma_2-\lambda) >0$ and
\begin{subequations}
\begin{align}
&(2\nu\sigma_1-\lambda)\|b\|^2 +(2\sigma_2-\lambda)\|d\|^2 =\kappa\|b\|^2 +\mu\|d\|^2 \\ 
&=\frac{1}{\kappa+\mu}\|\kappa b+\mu d\|^2
+\frac{\kappa\mu}{\kappa+\mu}\|b-d\|^2 \\
&=\frac{1}{2(\nu\sigma_1+\sigma_2-\lambda)}\|(2\nu\sigma_1-\lambda)b +(2\sigma_2-\lambda)d\|^2 +\frac{(2\nu\sigma_1-\lambda)(2\sigma_2-\lambda)}{2\eta^2(\nu\sigma_1+\sigma_2-\lambda)}\|a-f\|^2,
\end{align}
\end{subequations}
where the last equality is due to the fact that $b-d= \frac{1}{\eta}(a-f)$. Substituting into \eqref{e:abc'}, we get
\begin{align}
\|f\|^2 &\leq\|a\|^2 -\left(\frac{(2\nu\sigma_1-\lambda)(2\sigma_2-\lambda)}{2\eta\nu(\nu\sigma_1+\sigma_2-\lambda)}+\frac{\lambda}{\eta\nu}-\frac{\delta}{2\eta\nu\sigma_3} -1\right)\|a-f\|^2\notag \\ 
&\quad -\frac{\delta}{2\eta\nu\sigma_3}\|a-f -2\eta\sigma_3 c\|^2
-\frac{\eta}{2\nu(\nu\sigma_1+\sigma_2-\lambda)} \|(2\nu\sigma_1-\lambda)b +(2\sigma_2-\lambda)d\|^2,
\end{align}
which proves \eqref{e:lambda<}.

Finally, in both cases \ref{p:abs_lambda=} and \ref{p:abs_lambda<}, we have that
\begin{equation}
\|Tx-Ty\|^2\leq \|x-y\|^2 -\left(\frac{\eta^*}{\eta}-1\right)\|(\Id-T)x-(\Id-T)y\|^2,
\end{equation}
which implies that $T$ is conically $\frac{\eta}{\eta^*}$-averaged due to Proposition~\ref{p:coavg}\ref{p:coavg_ori}\&\ref{p:coavg_ine}.
\end{proof}

\begin{theorem}[abstract convergence]
\label{t:abscvg}
Let $T_1$, $T_2$, and $T_3$ be respectively $\sigma_1$-, $\sigma_2$-, and $\sigma_3$-cocoercive. Let $\eta, \nu,\lambda,\delta\in \RPP$ and define
\begin{equation}
T :=\Id -\eta T_1 +\eta T_2(-\nu\Id +\lambda T_1 -\delta T_3T_1).
\end{equation}
Suppose that $\Fix T\neq \varnothing$ and that either
\begin{enumerate}[label ={\rm (\alph*)}]
\item\label{a:lambda=} 
$\lambda =2\nu\sigma_1 =2\sigma_2$ and $\eta <\eta^* :=\frac{1}{\nu}\left(\lambda-\frac{\delta}{2\sigma_3}\right)$; or
\item\label{a:lambda<} 
$\lambda <\nu\sigma_1+\sigma_2$ and $\eta <\eta^* :=\frac{1}{\nu}\left(\frac{(2\nu\sigma_1-\lambda)(2\sigma_2-\lambda)}{2(\nu\sigma_1+\sigma_2-\lambda)}+\lambda-\frac{\delta}{2\sigma_3}\right)$.
\end{enumerate} 
Let $(x_n)_{n\in\NN}\subset \dom T$ be a sequence in generated by $T$ and set $S :=-\nu\Id +\lambda T_1 -\delta T_3T_1$. Then the following hold:
\begin{enumerate}
\item\label{t:abscvg_T}
$T$ is $\frac{\eta}{\eta^*}$-averaged. Consequently, $(x_n)_{n\in \mathbb{N}}$ converges weakly to a point $x^*\in \Fix T$ and the rate of asymptotic regularity of $T$ is $o(1/\sqrt{n})$, \emph{i.e.}, $\|x_n-Tx_n\| =o(1/\sqrt{n})$ as $n\to +\infty$.

\item\label{t:abscvg_T3T1} 
$(T_3T_1x_n)_{n\in \mathbb{N}}$ converges strongly to $T_3T_1x^*$ and $T_3T_1(\Fix T) =\{T_3T_1x^*\}$.

\item\label{t:abscvg_T1a}
If \ref{a:lambda=} holds and $\nu =\lambda-1$, then $(T_1x_n)_{n\in \mathbb{N}}$ and $(T_2Sx_n)_{n\in \mathbb{N}}$ converge weakly to $T_1x^* =T_2Sx^*$.  

\item\label{t:abscvg_T1b}
If \ref{a:lambda<} holds, then $(T_1x_n)_{n\in \mathbb{N}}$ and $(T_2Sx_n)_{n\in \mathbb{N}}$ converge strongly to $T_1x^* =T_2Sx^*$ and $T_1(\Fix T) = T_2S(\Fix T) =\{T_1x^*\}$.
\end{enumerate}
\end{theorem}
\begin{proof}
Set $S :=-\nu\Id +\lambda T_1 -\delta T_3T_1$, $\omega_1 :=\frac{\eta^*}{\eta}-1$, $\omega_2 :=\frac{\delta}{2\eta\nu\sigma_3}$, and 
\begin{equation}
\omega_3 :=\begin{cases}
0 &\text{if~} \lambda =2\nu\sigma_1 =2\sigma_2,\\
\frac{\eta}{2\nu(\nu\sigma_1+\sigma_2-\lambda)} &\text{if~} \lambda <\nu\sigma_1+\sigma_2. 
\end{cases}
\end{equation}
Then $\omega_1 >0$, $\omega_2 >0$, and $\omega_3 \geq 0$. 
We derive from Proposition~\ref{p:abs} that, for all $x,y\in \dom T$, 
\begin{align}\label{e:key}
\|Tx-Ty\|^2&\leq \|x-y\|^2 -\omega_1\|(\Id-T)x-(\Id-T)y\|^2 \notag\\
&\quad -\omega_2\|(\Id-T)x-(\Id-T)y-2\eta\sigma_3(T_3T_1x-T_3T_1y)\|^2 \notag\\
&\quad -\omega_3\|(2\nu\sigma_1-\lambda)(T_1x-T_1y) +(2\sigma_2-\lambda)(T_2Sx-T_2Sy)\|^2
\end{align}
and $T$ is conically $\frac{\eta}{\eta^*}$-averaged.

\ref{t:abscvg_T}: Since $\eta <\eta^*$, $T$ is $\frac{\eta}{\eta^*}$-averaged. By \cite[Corollary~2.10]{BDP19}, $(x_n)_{n\in \mathbb{N}}$ converges weakly to a point $x^*\in \Fix T$ and the rate of asymptotic regularity of $T$ is $o(1/\sqrt{n})$.  

\ref{t:abscvg_T3T1}: Let $y\in\Fix T$. It follows from \eqref{e:key} that, for all $n \in \mathbb{N}$,
\begin{align}
\|x_{n+1}-y\|^2&\leq \|x_n-y\|^2 -\omega_1\|(\Id-T)x_n\|^2 \notag\\
&\quad -\omega_2\|(\Id-T)x_n-2\eta\sigma_3(T_3T_1x_n-T_3T_1y)\|^2 \notag\\
&\quad -\omega_3\|(2\nu\sigma_1-\lambda)(T_1x_n-T_1y) +(2\sigma_2-\lambda)(T_2Sx_n-T_2Sy)\|^2.
\end{align}
Telescoping this inequality yields
\begin{align}\label{e:tabscvg1}
&\omega_1\sum_{n=0}^{\infty} \|(\Id-T)x_n\|^2 +\omega_2\sum_{n=0}^{\infty} \|(\Id-T)x_n-2\eta\sigma_3(T_3T_1x_n-T_3T_1y)\|^2 \notag\\
&+\omega_3\sum_{n=0}^{\infty} \|(2\nu\sigma_1-\lambda)(T_1x_n-T_1y) +(2\sigma_2-\lambda)(T_2Sx_n-T_2Sy)\|^2
\leq \|x_0-y\|^2 <+\infty.
\end{align}
Since $\omega_1,\omega_2 >0$ and $\omega_3 \geq 0$, we deduce that, as $n\to +\infty$,
\begin{equation}
(\Id-T)x_n\to 0 \text{~~and~~} (\Id-T)x_n-2\eta\sigma_3(T_3T_1x_n-T_3T_1y)\to 0,
\end{equation}
which imply that
\begin{equation}
T_3T_1x_n\to T_3T_1y.
\end{equation}
As $y$ is arbitrary in $\Fix T$ and $x^*\in \Fix T$, we must have $T_3T_1y =T_3T_1x^*$, and so $T_3T_1(\Fix T) =\{T_3T_1x^*\}$.

\ref{t:abscvg_T1a}: We will apply demiclosedness principle in Proposition~\ref{p:demi} to prove that $(T_1x_n)_{n\in \mathbb{N}}$ converges weakly to $T_1x^*$. First, recall from \ref{t:abscvg_T} that 
\begin{equation}\label{e:wkxk}
x_n\rightharpoonup x^*\in\Fix T.
\end{equation}
As a result, $(x_n)_{n\in \mathbb{N}}$ is bounded, and so is $(T_1x_n)_{n\in \mathbb{N}}$. Let $y^*$ be a weak cluster point of $(T_1x_n)_{n\in \mathbb{N}}$. Then there exists a subsequence $(x_{k_n})_{n\in \mathbb{N}}$ such that
\begin{equation}\label{e:wkT1}
T_1x_{k_n}\rightharpoonup y^*.
\end{equation} 
Define $z_n :=Sx_n =(1-\lambda)x_n +\lambda T_1x_n -\delta T_3T_1x_n$. Since $T_3T_1x_n\to T_3T_1x^*$ by \ref{t:abscvg_T3T1}, it follows that
\begin{equation}\label{e:wkzk}
z_{k_n}\rightharpoonup (1-\lambda) x^* +\lambda y^* -\delta T_3T_1x^* =:z^*.
\end{equation}
Next, we have from \ref{t:abscvg_T} that
\begin{equation}\label{e:wkT1-T2}
T_1x_{k_n} -T_2z_{k_n} =(T_1-T_2S)x_{k_n} =\frac{1}{\eta}(\Id-T)x_{k_n}\to 0,
\end{equation}
which, due to \eqref{e:wkT1}, implies that
\begin{equation}\label{e:wkT2}
T_2z_{k_n}\rightharpoonup y^*.
\end{equation}
Set $\rho_1 :=\lambda-1 =\nu >0$ and $\rho_2 :=1$. Then
\begin{equation}\label{e:wkrho}
\frac{\rho_1\sigma_1+\rho_2\sigma_2}{\rho_1+\rho_2}
=\frac{(\lambda-1)\frac{\lambda}{2\nu}+1\cdot\frac{\lambda}{2}}{\lambda}=1
\end{equation}
and it follows from \eqref{e:wkzk} that
\begin{equation}
\rho_1(x^*-y^*)+\rho_2(z^*-y^*)
=-\delta T_3T_1x^*.
\end{equation}
Using the definition of $z_n$ and then \eqref{e:wkT1-T2}, we obtain
\begin{align}\label{e:wksum}
&\rho_1(x_{k_n}-T_1x_{k_n}) +\rho_2(z_{k_n}-T_2z_{k_n})\notag\\
&=(\lambda-1)(x_{k_n}-T_1x_{k_n}) +(1-\lambda)x_{k_n} +\lambda T_1x_{k_n} -\delta T_3T_1x_{k_n} -T_2z_{k_n}\notag\\
&=T_1x_{k_n} -T_2z_{k_n} -\delta T_3T_1x_{k_n}\notag\\
&\to -\delta T_3T_1x^*=\rho_1(x^*-y^*) + \rho_2(z^*-T_2z^*).
\end{align}
Now, in view of \eqref{e:wkxk}, \eqref{e:wkT1}, \eqref{e:wkzk}, \eqref{e:wkT2}, \eqref{e:wkrho}, and \eqref{e:wksum}, we are ready to apply Proposition~\ref{p:demi} to derive that
\begin{equation}
y^* =T_1x^* =T_2z^*,
\end{equation}
which is the unique weak cluster point $(T_1x_n)_{n\in \mathbb{N}}$. Thus, $T_1x_n\rightharpoonup T_1x^*$. Since $T_1x_n-T_2Sx_n =\frac{1}{\eta}(\Id-T)x_n\to 0$ and $x^*\in \Fix T$, we derive that $T_2Sx_n\rightharpoonup T_1x^* =T_2Sx^*$.

\ref{t:abscvg_T1b}: In this case, $\omega_3>0$. So \eqref{e:tabscvg1} implies that, as $n\to +\infty$,
\begin{equation}\label{e:T1vsT2S}
(2\nu\sigma_1-\lambda)(T_1x_n-T_1y) +(2\sigma_2-\lambda)(T_2Sx_n-T_2Sy)\to 0.
\end{equation}
On the other hand,
\begin{equation}
(T_1x_n-T_1y) -(T_2Sx_n-T_2Sy) =\frac{1}{\eta}(\Id-T)x_n -\frac{1}{\eta}(\Id-T)y =\frac{1}{\eta}(\Id-T)x_n\to 0,
\end{equation}
which together with \eqref{e:T1vsT2S} yields
\begin{equation}
T_1x_n\to T_1y \text{~~and~~} T_2Sx_n\to T_2Sy. 
\end{equation}
Since $y$ is arbitrary in $\Fix T$ and $x^*\in \Fix T$, it also follows that $T_1y =T_1x^*$ and $T_2Sy =T_2Sx^* =T_1x^*$. Hence, $T_1(\Fix T) =T_2S(\Fix T) =\{T_1x^*\}$. The proof is complete.
\end{proof}

\section{Zeros of the sum of three operators}
\label{s:3opers}

In this section, we apply the result to the problem of finding a zero of the sum of three operators. We assume that the operator $A$ is maximally $\alpha$-monotone, the operator $B$ is maximally $\beta$-monotone, and the operator $C$ is $\sigma$-cocoercive. We will consider two cases: $\alpha+\beta=0$ and $\alpha+\beta>0$.

\begin{theorem}[convergence in the case $\alpha+\beta=0$]
\label{t:neutral}
Suppose that $A$ and $B$ are respectively maximally $\alpha$- and $\beta$-monotone with $\alpha+\beta =0$, and $C$ is $\sigma$-cocoercive. Let $\eta\in \RPP$ and let $\gamma\in \RPP$ be such that
\begin{equation}\label{e:tneutral-1}
1+2\gamma\alpha >0
\text{~~and~~} \eta^*
:=2+2\gamma\alpha-\frac{\gamma}{2\sigma} >0.
\end{equation}
Set $\delta =\frac{\gamma}{1+2\gamma\alpha}$, $\lambda =1+\frac{\delta}{\gamma}$, and let $(x_n)_{n\in\NN}$ be a sequence generated by $T_{A,B,C}$ in \eqref{e:TABC}. 
Then the following hold:
\begin{enumerate}
\item\label{t:neutral_dom} 
$T_{A,B,C}$ is single-valued and has full domain.
\item\label{t:neutral_avg} 
For all $x,y\in X$, 
\begin{align}
\|T_{A,B,C} x-T_{A,B,C} y\|^2 &\leq \|x-y\|^2 -\left(\frac{\eta^*}{\eta}-1\right)\|(\Id-T_{A,B,C})x-(\Id-T_{A,B,C})y\|^2 \notag\\
&\ -\frac{\gamma}{2\eta\sigma}\left\|(\Id-T_{A,B,C})x -(\Id-T_{A,B,C})y-2\eta\sigma(CJ_{\gamma A}x-CJ_{\gamma A}y)\right\|^2.
\end{align}
In particular, $T_{A,B,C}$ is conically $\frac{\eta}{\eta^*}$-averaged.
\item\label{t:neutral_cvg} 
If $\zer(A+B+C)\neq\varnothing$ and $\eta <\eta^*$, then the rate of asymptotic regularity of $T_{A,B,C}$ is $o(1/\sqrt{n})$ and $(x_n)_{n\in \mathbb{N}}$ converges weakly to a point $x^*\in \Fix T$, while $(J_{\gamma A}x_n)_{n\in \mathbb{N}}$ and $(J_{\delta B}Sx_n)_{n\in \mathbb{N}}$ converge weakly to $J_{\gamma A}x^* =J_{\delta B}Sx^*\in \zer(A+B+C)$ where $S :=(1-\lambda)\Id +\lambda J_{\gamma A} -\delta CJ_{\gamma A}$, $(CJ_{\gamma A}x_n)_{n\in \mathbb{N}}$ converges strongly to $CJ_{\gamma A}x^*$, and $C(\zer(A+B+C)) =\{CJ_{\gamma A}x^*\}$.
\end{enumerate}
\end{theorem}
\begin{proof}
First, we can check that there always exist $\gamma\in \RPP$ such that \eqref{e:tneutral-1} holds (indeed, by choosing $\gamma >0$ satisfying $1/\gamma >\max\{-2\alpha, -\alpha+1/(4\sigma)\}$). Next, we have that $1+\gamma\alpha =1/2 +(1+2\gamma\alpha)/2 >0$. Since $\alpha+\beta =0$, we also have
\begin{equation}\label{e:1+deltabeta}
1+\delta\beta =1-\delta\alpha =1-\frac{\gamma\alpha}{1+2\gamma\alpha} =\frac{1+\gamma\alpha}{1+2\gamma\alpha} >0.
\end{equation}
By Proposition~\ref{p:resol-mono}, $J_{\gamma A}$, $J_{\delta B}$, and hence $T_{A,B,C}$ are single-valued and have full domain. This proves \ref{t:neutral_dom}. 

Next, Proposition~\ref{p:resol-mono} also implies that $J_{\gamma A}$ and $J_{\delta B}$ are respectively $(1+\gamma\alpha)$- and $(1+\delta\beta)$-cocoercive. Set $\sigma_1 :=1+\gamma\alpha >0$, $\sigma_2 :=1+\delta\beta >0$, $\sigma_3 :=\sigma >0$, and $\nu :=\lambda-1 >0$. Then $2\nu\sigma_1 =2(\lambda-1)(1+\gamma\alpha) =2(1+\gamma\alpha)/(1+2\gamma\alpha) =\lambda$ and, by \eqref{e:1+deltabeta}, $2\sigma_2 =2(1+\delta\beta) =2(1+\gamma\alpha)/(1+2\gamma\alpha) =\lambda$. Also, 
\begin{equation}
\frac{1}{\nu}\left(\lambda-\frac{\delta}{2\sigma_3}\right) =\frac{1}{\lambda-1}\left(\lambda-\frac{(\lambda-1)\gamma}{2\sigma}\right) =1+\frac{\gamma}{\delta}-\frac{\gamma}{2\sigma} =\eta^* >0.
\end{equation}  
Applying Proposition~\ref{p:abs}\ref{p:abs_lambda=}, we get \ref{t:neutral_avg}.

Now, by Proposition~\ref{p:fix}, $J_{\gamma A}(\Fix T_{A,B,C}) =\zer(A+B+C)$. We then apply Theorem~\ref{t:abscvg}\ref{t:abscvg_T}--\ref{t:abscvg_T1a} to complete the proof.
\end{proof}

Using Theorem~\ref{t:neutral}, we recover the results in \cite[Theorem~2.1(1)]{DY17}, which partly spurred our interest in the topic.

\begin{corollary}
\label{c:2mono}
Suppose that $A$ and $B$ are respectively maximally monotone, that $C$ is $\sigma$-cocoercive. Let $\gamma\in \left]0, 4\sigma\right[$, $\eta\in \left]0, 2-\frac{\gamma}{2\sigma}\right[$, and define
\begin{equation}
T_{A,B,C} :=\Id -\eta J_{\gamma A} +\eta J_{\gamma B}(2J_{\gamma A} -\Id -\gamma CJ_{\gamma A}).
\end{equation}
Then the following hold:
\begin{enumerate}
\item\label{c:2mono_avg} 
$T_{A,B,C}$ is $\frac{2\eta\sigma}{4\sigma-\gamma}$-averaged.
\item\label{c:2mono_cvg} 
If $\zer(A+B+C)\neq\varnothing$, then the rate of asymptotic regularity of $T_{A,B,C}$ is $o(1/\sqrt{n})$ and $(x_n)_{n\in \mathbb{N}}$ converges weakly to a point $x^*\in \Fix T$, while $(J_{\gamma A}x_n)_{n\in \mathbb{N}}$ and $(J_{\gamma B}(2J_{\gamma A} -\Id -\gamma CJ_{\gamma A})x_n)_{n\in \mathbb{N}}$ converge weakly to $J_{\gamma A}x^* =J_{\gamma B}(2J_{\gamma A} -\Id -\gamma CJ_{\gamma A})x^*\in \zer(A+B+C)$, $(CJ_{\gamma A}x_n)_{n\in \mathbb{N}}$ converges strongly to $CJ_{\gamma A}x^*$, and $C(\zer(A+B+C)) =\{CJ_{\gamma A}x^*\}$.
\end{enumerate}
\end{corollary}
\begin{proof}
Apply Theorem~\ref{t:neutral} with $\alpha =\beta =0$, $\delta =\gamma$, and $\eta^* =2-\frac{\gamma}{2\sigma}$.
\end{proof}

\begin{remark}[range of parameter $\gamma$]
We note that while Corollary~\ref{c:2mono}\ref{c:2mono_avg} is straightforward from \cite[Proposition~2.1]{DY17}, Corollary~\ref{c:2mono}\ref{c:2mono_cvg} improves upon \cite[Theorem~2.1(1)]{DY17} by only requiring the parameter $\gamma\in \left]0, 4\sigma\right[$ instead of $\gamma\in \left]0, 2\sigma\varepsilon\right[$ with $\varepsilon\in \left]0, 1\right[$.
\end{remark}

Next, we consider the case $\alpha+\beta>0$. This case indeed allows for some flexibility in choosing the resolvent parameters $\gamma,\delta$. In particular, let us recall the case $\alpha+\beta =0$ in Theorem~\ref{t:neutral}, the resolvent parameters $\gamma,\delta$ must be directly related by
\begin{equation}
\delta=\frac{\gamma}{1+2\gamma\alpha},
\quad \text{or equivalently,}\quad 
\frac{1}{\delta}=\frac{1}{\gamma}+2\alpha.
\end{equation}
In the case $\alpha+\beta>0$, the above exact relation is no longer necessary; instead, for a given $\gamma$, one can choose $\delta$ within a range such that
\begin{equation}
\max\left\{0,\frac{1}{\gamma}+2\alpha-2\sqrt{\Delta}\right\} <\frac{1}{\delta} <\frac{1}{\gamma}+2\alpha+2\sqrt{\Delta}
\end{equation}
for some positive $\Delta$ that depends on $\alpha,\beta$ and $\gamma$ itself.
In the next results, we will show that such choices for $(\gamma,\delta)$ always exist, and will guarantee convergence of the algorithm.

\begin{lemma}[existence of resolvent parameters]
\label{l:exist}
Let $\alpha,\beta\in \RR$ be such that $\alpha+\beta>0$, let $\sigma\in \RPP$, and let $\gamma,\delta\in \RPP$. Set 
\begin{equation}
\gamma_0 :=\begin{cases}
0 &\text{if~} \alpha \geq \frac{1}{4\sigma},\\
-\alpha+\frac{1}{4\sigma} &\text{if~} -\frac{1}{4\sigma} \leq \alpha < \frac{1}{4\sigma},\\
2\beta -2\sqrt{(\alpha+\beta)(\beta-\frac{1}{4\sigma})} &\text{if~} \alpha < -\frac{1}{4\sigma}.
\end{cases}
\end{equation}
Then $\gamma_0\geq \max\{0, -\alpha+\frac{1}{4\sigma}\}$ and the following statements are equivalent:
\begin{enumerate}[label ={\rm (\roman*)}]
\item\label{l:exist_ori} 
$\frac{4\gamma\delta(1+\gamma\alpha)(1+\delta\beta)-(\gamma+\delta)^2}{2\gamma\delta^2(\alpha+\beta)}-\frac{\gamma}{2\sigma} >0$.

\item\label{l:exist_sol}
$\frac{1}{\gamma} >\gamma_0$ and $\max\{0, \frac{1}{\gamma}+2\alpha-2\sqrt{\Delta}\} <\frac{1}{\delta} <\frac{1}{\gamma}+2\alpha+2\sqrt{\Delta}$, where $\Delta :=(\alpha+\beta)(\frac{1}{\gamma}+\alpha-\frac{1}{4\sigma})$.
\end{enumerate}
Consequently, there always exist $\gamma,\delta\in \RPP$ that satisfy both \ref{l:exist_ori} and \ref{l:exist_sol}.
\end{lemma}
\begin{proof}
If $\alpha \geq -\frac{1}{4\sigma}$, then $\gamma_0 =\max\{0, -\alpha+\frac{1}{4\sigma}\}$ by definition. If $\alpha < -\frac{1}{4\sigma} <0$, then $\beta-\frac{1}{4\sigma} >\beta+\alpha >0$ and 
\begin{subequations}
\begin{align}
\gamma_0 &=2\beta -2\sqrt{(\alpha+\beta)(\beta-\tfrac{1}{4\sigma})} =\left(\sqrt{\beta-\tfrac{1}{4\sigma}} -\sqrt{\alpha+\beta}\right)^2 -\alpha +\tfrac{1}{4\sigma}\\ 
&\geq -\alpha +\tfrac{1}{4\sigma} =\max\{0, -\alpha +\tfrac{1}{4\sigma}\}.
\end{align}
\end{subequations}

Next, we have that
\begin{subequations}\label{e:equiv}
\begin{align}
& \frac{4\gamma\delta(1+\gamma\alpha)(1+\delta\beta)-(\gamma+\delta)^2}{2\gamma\delta^2(\alpha+\beta)}-\frac{\gamma}{2\sigma} >0\label{e:equiv-a}\\
\iff\ & 
(\gamma+\delta)^2 <4\gamma\delta(1+\gamma\alpha)(1+\delta\beta) -\frac{\gamma^2\delta^2(\alpha+\beta)}{\sigma}\\
\iff\ &
\left(1-4\gamma\beta-4\gamma^2\alpha\beta+\frac{\gamma^2(\alpha+\beta)}{\sigma}\right)\delta^2 -2\gamma(1+2\gamma\alpha)\delta +\gamma^2 <0\\
\iff\ &
\left(\frac{1}{\gamma^2}-4\beta\frac{1}{\gamma}-4\alpha\beta+\frac{\alpha+\beta}{\sigma}\right) -2\left(\frac{1}{\gamma}+2\alpha\right)\frac{1}{\delta} +\frac{1}{\delta^2} <0\\
\iff\ &
\Delta = (\alpha+\beta)(\frac{1}{\gamma}  +\alpha-\frac{1}{4\sigma})>0 \text{~and~} \frac{1}{\gamma}+2\alpha-2\sqrt{\Delta} <\frac{1}{\delta} <\frac{1}{\gamma}+2\alpha+2\sqrt{\Delta}.\\
\iff\ &
\frac{1}{\gamma} > -\alpha+\frac{1}{4\sigma} \text{~and~} \frac{1}{\gamma}+2\alpha-2\sqrt{\Delta} <\frac{1}{\delta} <\frac{1}{\gamma}+2\alpha+2\sqrt{\Delta}. \label{e:sol} 
\end{align}
\end{subequations}

Suppose \ref{l:exist_sol} holds, then $\frac{1}{\gamma}>\gamma_0\geq \max\{0, -\alpha+\frac{1}{4\sigma}\}$. So \eqref{e:sol} holds. It follows that \eqref{e:equiv-a} holds, which is  \ref{l:exist_ori}.

Now, suppose that \ref{l:exist_ori} holds. Then \eqref{e:sol} holds, and so $\frac{1}{\gamma} >\max\{0, -\alpha+\frac{1}{4\sigma}\}$ and $\frac{1}{\gamma}+2\alpha+2\sqrt{\Delta} >0$. To obtain \ref{l:exist_sol}, it suffices to show that $\frac{1}{\gamma} >\gamma_0$. If $\alpha \geq -\frac{1}{4\sigma}$, then $\gamma_0 =\max\{0, -\alpha+\frac{1}{4\sigma}\}$ and we readily have $\frac{1}{\gamma} >\gamma_0$. Let us consider the case when $\alpha < -\frac{1}{4\sigma}$. Then $\beta-\frac{1}{4\sigma} >\beta+\alpha >0$ and 
\begin{subequations}
\begin{align}
\frac{1}{\gamma}+2\alpha+2\sqrt{\Delta} >0 
&\iff \left(\sqrt{\alpha+\beta} +\sqrt{\tfrac{1}{\gamma}+\alpha-\tfrac{1}{4\sigma}}\right)^2 >\beta-\tfrac{1}{4\sigma} \\
&\iff \sqrt{\tfrac{1}{\gamma}+\alpha-\tfrac{1}{4\sigma}} >\sqrt{\beta-\tfrac{1}{4\sigma}} -\sqrt{\alpha+\beta} \\
&\iff \tfrac{1}{\gamma}+\alpha-\tfrac{1}{4\sigma} >\left(\sqrt{\beta-\tfrac{1}{4\sigma}} -\sqrt{\alpha+\beta}\right)^2 \\
&\iff \tfrac{1}{\gamma} > 2\beta -2\sqrt{(\alpha+\beta)(\beta-\tfrac{1}{4\sigma})}=\gamma_0,
\end{align}
\end{subequations}  
which finish our claim.

To see the existence of $\gamma$ and $\delta$, we choose $\gamma >0$ such that $\frac{1}{\gamma} >\gamma_0$ and then choose $\delta >0$ that satisfies the second condition in \ref{l:exist_sol}.
\end{proof}

We are now ready to prove the convergence of the algorithm for the case $\alpha+\beta>0$.

\begin{theorem}[convergence in the case $\alpha+\beta>0$]
\label{t:strong}
Suppose that $A$ and $B$ are respectively maximally $\alpha$- and $\beta$-monotone with $\alpha+\beta >0$, that $C$ is $\sigma$-cocoercive, and that $\gamma,\delta\in \RPP$ satisfy  
\begin{equation}\label{e:para}
\eta^* :=\frac{4\gamma\delta(1+\gamma\alpha)(1+\delta\beta)-(\gamma+\delta)^2}{2\gamma\delta^2(\alpha+\beta)}-\frac{\gamma}{2\sigma} >0. 
\end{equation}
Set $\lambda =1+\frac{\delta}{\gamma}$ and let $\eta\in \RPP$.
Let $(x_n)_{n\in\NN}$ be a sequence generated by $T_{A,B,C}$ in \eqref{e:TABC} and set $S :=(1-\lambda)\Id +\lambda J_{\gamma A} -\delta CJ_{\gamma A}$. 
Then the following hold:
\begin{enumerate}
\item\label{t:strong_dom} 
$T_{A,B,C}$ is single-valued and has full domain.
\item\label{t:strong_avg} 
For all $x,y\in X$, 
\begin{align}
\|T_{A,B,C} x&-T_{A,B,C} y\|^2 \leq \|x-y\|^2 -\left(\frac{\eta^*}{\eta}-1\right)\|(\Id-T_{A,B,C})x-(\Id-T_{A,B,C})y\|^2 \notag\\
&\ -\frac{\gamma}{2\eta\sigma}\left\|(\Id-T_{A,B,C})x -(\Id-T_{A,B,C})y-2\eta\sigma(CJ_{\gamma A}x-CJ_{\gamma A}y)\right\|^2 \notag\\
&\ -\frac{\gamma\eta}{2\delta^2(\alpha+\beta)} \|(\lambda-2+2\delta\alpha)(T_1x-T_1y) +(2-\lambda+2\delta\beta)(T_2Sx-T_2Sy)\|^2.
\end{align}
In particular, $T_{A,B,C}$ is conically $\frac{\eta}{\eta^*}$-averaged.
\item\label{t:strong_cvg} 
If $\zer(A+B+C)\neq\varnothing$ and $\eta <\eta^*$, then the rate of asymptotic regularity of $T_{A,B,C}$ is $o(1/\sqrt{n})$ and $(x_n)_{n\in \mathbb{N}}$ converges weakly to a point $x^*\in \Fix T$, while $(J_{\gamma A}x_n)_{n\in \mathbb{N}}$ and $(J_{\delta B}Sx_n)_{n\in \mathbb{N}}$ converge strongly to $J_{\gamma A}x^* =J_{\delta B}Sx^*\in \zer(A+B+C)$, $(CJ_{\gamma A}x_n)_{n\in \mathbb{N}}$ converges strongly to $CJ_{\gamma A}x^*$, and $\zer(A+B+C) =\{J_{\gamma A}x^*\}$.
\end{enumerate}
\end{theorem}
\begin{proof}
First, Lemma~\ref{l:exist} ensures the existence of $\gamma,\delta\in \RPP$ satisfying \eqref{e:para}. In view of \eqref{e:para}, it also follows from Lemma~\ref{l:exist} that $1/\gamma >-\alpha+\frac{1}{4\sigma}$, and so $1+\gamma\alpha >\gamma/(4\sigma) >0$, which together with \eqref{e:para} implies that $1+\delta\beta >0$. In turn, Proposition~\ref{p:resol-mono} implies that $J_{\gamma A}$, $J_{\delta B}$, and hence $T_{A,B,C}$ are single-valued and have full domain, and we get \ref{t:strong_dom}.

We also derive from Proposition~\ref{p:resol-mono} that $J_{\gamma A}$ and $J_{\delta B}$ are $(1+\gamma\alpha)$- and $(1+\delta\beta)$-cocoercive, respectively. Now, set $\sigma_1 :=1+\gamma\alpha >0$, $\sigma_2 :=1+\delta\beta >0$, and $\sigma_3 :=\sigma >0$, and $\nu :=\lambda-1 >0$. On the one hand, since $\alpha+\beta >0$, 
\begin{equation}
\nu\sigma_1+\sigma_2 =(\lambda-1)(1+\gamma\alpha) +(1+\delta\beta) =\lambda +\delta(\alpha+\beta) >\lambda.
\end{equation}
On the other hand,
\begin{subequations}
\begin{align}
&\frac{1}{\nu}\left(\frac{(2\nu\sigma_1-\lambda)(2\sigma_2-\lambda)}{2(\nu\sigma_1+\sigma_2-\lambda)}+\lambda-\frac{\delta}{2\sigma_3}\right) \notag\\ &=\frac{1}{\nu}\left(\frac{4\nu\sigma_1\sigma_2-\lambda^2}{2(\nu\sigma_1+\sigma_2-\lambda)}-\frac{\delta}{2\sigma_3}\right)\\
&=\frac{\gamma}{\delta}\left(\frac{4\gamma\delta(1+\gamma\alpha)(1+\delta\beta)-(\gamma+\delta)^2}{2\gamma^2\delta(\alpha+\beta)}-\frac{\delta}{2\sigma}\right)\\
&=\frac{4\gamma\delta(1+\gamma\alpha)(1+\delta\beta)-(\gamma+\delta)^2}{2\gamma\delta^2(\alpha+\beta)}-\frac{\gamma}{2\sigma} =\eta^* >0.
\end{align}
\end{subequations}
Therefore, we obtain \ref{t:strong_avg} due to Proposition~\ref{p:abs}\ref{p:abs_lambda<}.

Finally, applying Theorem~\ref{t:abscvg}\ref{t:abscvg_T},\,\ref{t:abscvg_T3T1}\&\ref{t:abscvg_T1b} and noting that $J_{\gamma A}(\Fix T_{A,B,C}) =\zer(A+B+C)$ due to Proposition~\ref{p:fix}, we complete the proof.
\end{proof}

\section{Zeros of the sum of two operators}
\label{s:2opers}

The new results in Theorems~\ref{t:neutral} and \ref{t:strong} allow us to revisit the relaxed forward-backward, relaxed backward-forward, and adaptive Douglas--Rachford algorithms for solving the problem of finding a zero of the sum of two operators.

\begin{theorem}[relaxed forward-backward]
\label{t:FB}
Suppose that $B$ is maximally $\beta$-monotone with $\beta\in \RP$ and that $C$ is $\sigma$-cocoercive. Let $\gamma\in \left]0, 4\sigma\right[$, $\eta\in \left]0, 2-\frac{\gamma}{2\sigma}\right[$, and let $(x_n)_{n\in\NN}$ be a sequence generated by 
\begin{equation}
T_{\rm FB} :=(1-\eta)\Id +\eta J_{\gamma B}(\Id -\gamma C).
\end{equation}
Then the following hold:
\begin{enumerate}
\item\label{t:FB_avg} 
For all $x,y\in X$, 
\begin{align}
\|T_{\rm FB} x-T_{\rm FB} y\|^2 &\leq \|x-y\|^2 -\left(\frac{4\sigma-\gamma}{2\eta\sigma}-1\right)\|(\Id-T_{\rm FB})x-(\Id-T_{\rm FB})y\|^2 \notag\\
&\quad -\frac{\gamma}{2\eta\sigma}\left\|(\Id-T_{\rm FB})x -(\Id-T_{\rm FB})y-2\eta\sigma(Cx-Cy)\right\|^2.
\end{align}
In particular, $T_{\rm FB}$ is $\frac{2\eta\sigma}{4\sigma-\gamma}$-averaged.
\item\label{t:FB_cvg} 
If $\zer(B+C)\neq\varnothing$, then the rate of asymptotic regularity of $T_{\rm FB}$ is $o(1/\sqrt{n})$ and $(x_n)_{n\in \mathbb{N}}$ converges weakly to a point $x^*\in \zer(B+C)$, while $(Cx_n)_{n\in \mathbb{N}}$ converges strongly to $Cx^*$, and $C(\zer(B+C)) =\{Cx^*\}$. Moreover, if additionally $\beta >0$, then $(x_n)_{n\in \mathbb{N}}$ converges strongly to $x^*$ and $\zer(B+C) =\{x^*\}$.
\end{enumerate}
\end{theorem}
\begin{proof}
Apply Theorems~\ref{t:neutral} and \ref{t:strong} with $A =0$, $\alpha =0$, $\lambda =2$, and $\delta =\gamma$. 
\end{proof}

\begin{theorem}[relaxed backward-forward]
\label{t:BF}
Suppose that $A$ is maximally $\alpha$-monotone with $\alpha\in \RP$ and that $C$ is $\sigma$-cocoercive. Let $\gamma\in \left]0, 4\sigma\right[$, $\eta\in \left]0, 2-\frac{\gamma}{2\sigma}\right[$, and let $(x_n)_{n\in\NN}$ be a sequence generated by 
\begin{equation}
T_{\rm BF} :=(1-\eta)\Id +\eta (\Id -\gamma C)J_{\gamma A}.
\end{equation}
Then the following hold:
\begin{enumerate}
\item\label{t:BF_avg} 
For all $x,y\in X$, 
\begin{align}
\|T_{\rm BF} x-T_{\rm BF} y\|^2 &\leq \|x-y\|^2 -\left(\frac{4\sigma-\gamma}{2\eta\sigma}-1\right)\|(\Id-T_{\rm BF})x-(\Id-T_{\rm BF})y\|^2 \notag\\
&\quad -\frac{\gamma}{2\eta\sigma}\left\|(\Id-T_{\rm BF})x -(\Id-T_{\rm BF})y-2\eta\sigma(Cx-Cy)\right\|^2.
\end{align}
In particular, $T_{\rm BF}$ is $\frac{2\eta\sigma}{4\sigma-\gamma}$-averaged.
\item\label{t:BF_cvg} 
If $\zer(A+C)\neq\varnothing$, then the rate of asymptotic regularity of $T_{\rm BF}$ is $o(1/\sqrt{n})$ and $(x_n)_{n\in \mathbb{N}}$ converges weakly to a point $x^*\in \Fix T_{\rm BF}$, while $(J_{\gamma A}x_n)_{n\in \mathbb{N}}$ converges weakly to $J_{\gamma A}x^*\in \zer(A+C)$, $(Cx_n)_{n\in \mathbb{N}}$ converges strongly to $Cx^*$, and $C(\zer(A+C)) =\{Cx^*\}$. Moreover, if additionally $\alpha >0$, then $(J_{\gamma A}x_n)_{n\in \mathbb{N}}$ converges strongly to $J_{\gamma A}x^*\in \zer(A+C)$ and $\zer(A+C) =\{J_{\gamma A}x^*\}$.
\end{enumerate}
\end{theorem}
\begin{proof}
Apply Theorems~\ref{t:neutral} and \ref{t:strong} with $B =0$, $\beta =0$, $\lambda =2$, and $\delta =\gamma$. 
\end{proof}

\begin{theorem}[adaptive DR]
\label{t:DR}
Suppose that $A$ and $B$ are respectively maximally $\alpha$- and $\beta$-monotone, that either
\begin{enumerate}[label ={\rm (\alph*)}]
\item\label{a:=0} 
$\alpha+\beta =0$, $1+2\gamma\alpha >0$, $\delta =\frac{\gamma}{1+2\gamma\alpha}$, $\eta^* =2$; or 
\item\label{a:>0} 
$\alpha+\beta >0$, $\eta^* :=\frac{4\gamma\delta(1+\gamma\alpha)(1+\delta\beta)-(\gamma+\delta)^2}{2\gamma\delta^2(\alpha+\beta)} >0$.
\end{enumerate}
Let $\lambda =1+\frac{\delta}{\gamma}$, $\eta\in \left]0, \eta^*\right[$, and let $(x_n)_{n\in\NN}$ be a sequence generated by 
\begin{equation}
T_{\rm DR} :=\Id -\eta J_{\gamma A} +\eta J_{\delta B}((1-\lambda)\Id +\lambda J_{\gamma A}).
\end{equation}
Set $S :=(1-\lambda)\Id +\lambda J_{\gamma A}$. Then the following hold:
\begin{enumerate}
\item\label{t:DR_avg} 
$T_{\rm DR}$ is $\frac{\eta}{\eta^*}$-averaged and has full domain.
\item\label{t:DR_cvg} 
If $\zer(A+B)\neq\varnothing$, then the rate of asymptotic regularity of $T_{\rm DR}$ is $o(1/\sqrt{n})$ and $(x_n)_{n\in \mathbb{N}}$ converges weakly to a point $x^*\in \Fix T$ with $J_{\gamma A}x^*\in \zer(A+B)$. Moreover, when \ref{a:=0} holds, $(J_{\gamma A}x_n)_{n\in \mathbb{N}}$ and $(J_{\delta B}Sx_n)_{n\in \mathbb{N}}$ converge weakly to $J_{\gamma A}x^* =J_{\delta B}Sx^*$; when \ref{a:>0} holds, $(J_{\gamma A}x_n)_{n\in \mathbb{N}}$ and $(J_{\delta B}Sx_n)_{n\in \mathbb{N}}$ converge strongly to $J_{\gamma A}x^* =J_{\delta B}Sx^*$ and $\zer(A+B) =\{J_{\gamma A}x^*\}$.
\end{enumerate}
\end{theorem}
\begin{proof}
Apply Theorems~\ref{t:neutral} and \ref{t:strong} with $C =0$ and note that the operator $C$ is $\sigma$-cocoercive with any $\sigma >0$.
\end{proof}

\section{Minimizing the sum of three functions}
\label{s:min3f}

In this section, we consider the problem of minimizing the sum of three functions. Let $f\colon X\to \left]-\infty,+\infty\right]$. Then $f$ is \emph{proper} if $\dom f :=\menge{x\in X}{f(x) <+\infty}\neq \varnothing$, and \emph{lower semicontinuous} if $\forall x\in X$, $f(x)\leq \liminf_{z\to x} f(z)$. 
Given $\alpha\in \RR$, the function $f$ is \emph{$\alpha$-convex} if $\forall x,y\in\dom f$, $\forall\kappa\in \left]0,1\right[$,
\begin{equation}
\label{e:alpha-cvx}
f((1-\kappa) x+\kappa y) +\frac{\alpha}{2}\kappa(1-\kappa)\|x-y\|^2\leq (1-\kappa)f(x)+\kappa f(y).
\end{equation}
We simply say $f$ is \emph{convex} if $\alpha=0$.
We also say that $f$ is \emph{strongly convex} or \emph{weakly convex}, if $\alpha >0$ or $\alpha<0$, respectively. 

Next, let $f:X\to\left]-\infty,+\infty\right]$ be proper. The Fr\'echet subdifferential of $f$ at $x$ is defined by 
\begin{equation}
\label{e:Frechet}
\widehat{\partial}f(x) :=\Menge{u\in X}{\liminf_{z\to x}\frac{f(z)-f(x)-\scal{u}{z-x}}{\|z-x\|}\geq 0}.
\end{equation}
The \emph{proximity operator} of $f$ with parameter $\gamma\in \RPP$ is the mapping $\prox_{\gamma f}\colon X\rightrightarrows X$ defined by
\begin{equation}
\label{e:prox}
\forall x\in X, \quad \prox_{\gamma f}(x) :=\argmin_{z\in X} \left(f(z)+\frac{1}{2\gamma}\|z-x\|^2\right).
\end{equation}
We refer to \cite{CP11} for a list of proximity operators of common convex functions. For an $\alpha$-convex function, the relationship between its Fr\'echet subdifferential and its proximity operator is described in the following lemma.
\begin{lemma}[proximity operators of $\alpha$-convex functions]
\label{l:prox}
Let $f\colon X\to \left]-\infty,+\infty\right]$ be a proper, lower semicontinuous and $\alpha$-convex function. Let $\gamma\in \RPP$ be such that $1+\gamma\alpha >0$. Then
\begin{enumerate}
\item\label{l:prox_diff} 
$\widehat{\partial}f$ is maximally $\alpha$-monotone.
\item\label{l:prox_equal} 
$\prox_{\gamma f} =J_{\gamma\widehat{\partial}f}$ is single-valued and has full domain.
\end{enumerate}
\end{lemma}
\begin{proof}
See~\cite[Lemma~5.2]{DP19}.
\end{proof}

Now, we assume that $f,g:X\to\left]-\infty,+\infty\right]$ are proper lower semiconinuous, and respectively $\alpha$- and $\beta$-convex functions, and $h:X\to\RR$ is a differentiable convex function with Lipschitz continuous gradient. We will solve the minimization problem 
\begin{equation}
\label{e:minsum}
\min_{x\in X}\quad f(x)+g(x)+h(x)
\end{equation}
by employing the operator
\begin{equation}\label{e:T_prox}
T_{f,g,h} :=\Id -\eta \prox_{\gamma f} +\eta \prox_{\delta g}((1-\lambda)\Id+\lambda\prox_{\gamma f}-\delta\nabla h \prox_{\gamma f}),
\end{equation}
with appropriately chosen parameters $\gamma,\delta,\lambda,\eta\in\RPP$. 

\begin{theorem}[minimizing the sum of three functions]
\label{t:f-cvg}
Let $f,g\colon X\to \left]-\infty,+\infty\right]$ be proper lower semicontinuous functions and let $h\colon X\to \mathbb{R}$ be a differentiable convex function whose gradient is Lipschitz continuous with constant $1/\sigma$. Suppose that $f$ and $g$ are $\alpha$-convex and $\beta$-convex, respectively, and that either
\begin{enumerate}[label ={\rm (\alph*)}]
\item\label{a:=0'}
$\alpha+\beta =0$, $1+2\gamma\alpha >0$, $\delta =\frac{\gamma}{1+2\gamma\alpha}$, $\eta^* :=2+2\gamma\alpha-\frac{\gamma}{2\sigma}$; or
\item\label{a:>0'}
$\alpha+\beta >0$, $\eta^* :=\frac{4\gamma\delta(1+\gamma\alpha)(1+\delta\beta)-(\gamma+\delta)^2}{2\gamma\delta^2(\alpha+\beta)}-\frac{\gamma}{2\sigma} >0$.
\end{enumerate}
Set $\lambda =1 +\frac{\delta}{\gamma}$ and $S :=(1-\lambda)\Id+\lambda\prox_{\gamma f}-\delta\nabla h \prox_{\gamma f}$. Let $(x_n)_\nnn$ be a sequence generated by $T_{f,g,h}$ in \eqref{e:T_prox}. Then the following hold:
\begin{enumerate}
\item
$T_{f,g,h}$ is conically $\frac{\eta}{\eta^*}$-averaged and has full domain.
\item
If $\zer(\widehat{\partial} f+\widehat{\partial} g+\nabla h)\neq \varnothing$ and $\eta <\eta^*$, then the rate of asymptotic regularity of $T_{f,g,h}$ is $o(1/\sqrt{n})$ and $(x_n)_{n\in \mathbb{N}}$ converges weakly to a point $x^*\in \Fix T_{f,g,h}$ with 
\begin{equation}
\prox_{\gamma f}x^* \in \zer(\widehat{\partial} f+\widehat{\partial} g+\nabla h)\subseteq \argmin(f+g+h),
\end{equation}
while $(\nabla h \prox_{\gamma f}x_n)_{n\in \mathbb{N}}$ converges strongly to $\nabla h \prox_{\gamma f}x^*$ and $\nabla h(\zer(\widehat{\partial} f+\widehat{\partial} g+\nabla h)) =\{\nabla h \prox_{\gamma f}x^*\}$.
Moreover, when \ref{a:=0'} holds, $(\prox_{\gamma f}x_n)_{n\in \mathbb{N}}$ and $(\prox_{\delta g}Sx_n)_{n\in \mathbb{N}}$ converge weakly to $\prox_{\gamma f}x^* =\prox_{\delta g}Sx^*$; when \ref{a:>0'} holds, $(\prox_{\gamma f}x_n)_{n\in \mathbb{N}}$ and $(\prox_{\delta g}Sx_n)_{n\in \mathbb{N}}$ converge strongly to $\prox_{\gamma f}x^* =\prox_{\delta g}Sx^*$ and $\zer(\widehat{\partial} f+\widehat{\partial} g+\nabla h) =\{\prox_{\gamma f}x^*\}$.
\end{enumerate}
\end{theorem}
\begin{proof}
As in the proofs of Theorems~\ref{t:neutral} and \ref{t:strong}, we have that $1+\gamma\alpha >0$ and $1+\delta\beta >0$. By Lemma~\ref{l:prox}, $\widehat{\partial}f$ and $\widehat{\partial}g$ are maximally $\alpha$-monotone and $\beta$-monotone, respectively, and $\prox_{\gamma f} =J_{\gamma\widehat{\partial}f}$ and $\prox_{\gamma g} =J_{\gamma\widehat{\partial}g}$. By \cite[Theorem~18.15(i)\&(v)]{BC17}, $\nabla h$ is $\sigma$-cocoercive. In addition, from Proposition~\ref{p:fix} and \cite[Lemma~5.3]{DP19}, we obtain the relationship between the fixed points of $T_{f,g,h}$ and the minimizers of \eqref{e:minsum}
\begin{equation}
\prox_{\gamma f}(\Fix T_{f,g,h})=\zer(\widehat{\partial} f+\widehat{\partial} g+\nabla h)\subseteq \argmin(f+g+h).
\end{equation}
The conclusion then follows by applying Theorems~\ref{t:neutral} and \ref{t:strong} to $A =\widehat{\partial}f$, $B =\widehat{\partial}g$, and $C =\nabla h$.
\end{proof}

\begin{remark}[minimizing the sum of two functions]
Analogous to Section~\ref{s:2opers}, one can apply Theorem~\ref{t:f-cvg} with $f=0$, $g=0$, or $h=0$ to obtain corresponding algorithms for minimizing the sum of two functions.
\end{remark}

\noindent\textbf{Acknowledgement:}
HMP was partially supported by Autodesk, Inc. via a gift made to the Department of Mathematical Sciences, UMass Lowell.

\end{document}